\title{On the group of automorphisms \\
of Horikawa surfaces}

\author{Vicente Lorenzo}
\date{}
\RequirePackage{fix-cm}
\RequirePackage{amsmath}

\documentclass{article}
\usepackage{graphicx}
\usepackage[T1]{fontenc}
\usepackage[utf8]{inputenc}
\usepackage{lmodern}
\usepackage[english]{babel}
\usepackage{amssymb}
\usepackage{amsthm}
\usepackage[all]{xy}
\usepackage{enumerate}
\usepackage[none]{hyphenat}
\usepackage{etoolbox}
\usepackage{microtype}
\usepackage{lipsum}
\usepackage[pdfencoding=auto, psdextra]{hyperref}
\usepackage{tikz-cd}

\newtheorem{theorem}{Theorem}
\newtheorem{lemma}{Lemma}
\newtheorem{proposition}{Proposition}

\newtheorem{corollary}{Corollary}
\theoremstyle{remark}\newtheorem{remark}{Remark}
\theoremstyle{remark}
\makeatletter
\renewenvironment{proof}[1][\proofname]{%
  \par\pushQED{\qed}\normalfont%
  \topsep6\p@\@plus6\p@\relax
  \trivlist\item[\hskip\labelsep\bfseries#1\@addpunct{.}]%
  \ignorespaces
}{%
  \popQED\endtrivlist\@endpefalse
}
\makeatother
\newenvironment{acknowledgements}{\textit{Acknowledgements.}}{}

\newcommand\blfootnote[1]{%
  \begingroup
  \renewcommand\thefootnote{}\footnote{#1}%
  \addtocounter{footnote}{-1}%
  \endgroup
}

\begin{document}

\maketitle

\begin{abstract}
Minimal algebraic surfaces of general type 
 $X$
  such that $K^2_X=2\chi(\mathcal{O}_X)-6$ are called Horikawa surfaces.
 In this note the group of automorphisms of Horikawa surfaces is studied. The main result states that given an admissible pair $(K^2, \chi)$ such that $K^2=2\chi-6$, every irreducible component of Gieseker's moduli space    
  $\mathfrak{M}_{K^2,\chi}$ contains
an open subset consisting of
surfaces with group of automorphisms isomorphic to 
 $\mathbb{Z}_2$.
\end{abstract}

\section{Introduction.}\label{Intro}

We work over the field $\mathbb{C}$ of complex numbers. The main numerical invariants of an algebraic surface $X$ are the self-intersection of its canonical class $K^2_X$ and its holomorphic Euler characteristic $\chi(\mathcal{O}_X)$. 
 If $X$ is minimal and of general type, the following inequalities are well known to be satisfied (cf. \cite[Chapter VII]{Barth2004}):
\begin{equation}\label{GeneralType}
 \chi(\mathcal{O}_X)\geq 1,\quad K_X^2\geq 1, \quad 2\chi(\mathcal{O}_X)-6\leq K_X^2\leq 9\chi(\mathcal{O}_X).
\end{equation}
Minimal algebraic surfaces of general type $X$   such that $K^2_X=2\chi(\mathcal{O}_X)-6$ were already studied by Enriques \cite{EnriquesPaper}, \cite[Section VIII.11]{EnriquesBook} but they are often called Horikawa surfaces because of 
Horikawa's contribution to their deformation theory \cite{Hor1}. One property of Horikawa surfaces is that their 
canonical system is base-point-free and induces a morphism  whose image is a surface. Moreover, this morphism has degree 2 (see Theorem \ref{StructureEvenHorikawa}). 
As a consequence, every Horikawa surface $X$ has a $\mathbb{Z}_2$-action generated by the involution that 
 sends a general point of $X$ to the point with the same image via the canonical map of $X$. In particular, the group of automorphisms $\textrm{Aut}(X)$ of $X$ has a subgroup isomorphic to $\mathbb{Z}_2$ and one may wonder whether this subgroup is proper or not when $X$ is sufficiently general.
 
 It is known that under some ampleness and generality assumptions, the group of automorphisms of a surface that can be realized as an abelian $G$-cover is precisely the group $G$ (see \cite{FanPar} or \cite{Manetti1997}). 
 Since Horikawa surfaces can be realized as $\mathbb{Z}_2$-covers via their canonical map, one could expect the group of automorphisms of a general Horikawa surface to be $\mathbb{Z}_2$. Nevertheless, Horikawa surfaces do not satisfy the hypothesis of 
 \cite[Theorem 4.6]{FanPar} nor \cite[Corollary B]{Manetti1997}
 and, to the best of the author's knowledge, there is no reference allowing to prove this in a straightforward way.
 
 Given an admissible pair $(K^2, \chi)$, i.e. a pair of integers satisfying the inequalities (\ref{GeneralType}), let us denote by $\mathfrak{M}_{K^2,\chi}$ Gieseker’s moduli space of canonical
models of surfaces of general type
with fixed self-intersection of the canonical class $K^2$ and fixed holomorphic Euler characteristic $\chi$. The aim of this note is to prove the following:
 
 \begin{theorem}\label{AutomorphismsEvenHorikawa}
Let $(K^2, \chi)$ be an admissible pair such that $K^2=2\chi-6$. Then every irreducible component of $\mathfrak{M}_{K^2, \chi}$
contains an open subset consisting of surfaces with group of automorphisms isomorphic to $\mathbb{Z}_2$.
\end{theorem}

The proof of Theorem \ref{AutomorphismsEvenHorikawa} is based on the following idea. Let $X$ be a smooth surface with an involution $\tau$ that induces a $\mathbb{Z}_2$-cover $X\to X/\langle \tau\rangle$ with building data $\{L,B\}$ (see Section \ref{SectionSimpleCyclic}). Then every automorphism $\sigma$ of $X$ that commutes with $\tau$  induces an automorphism $\overline{\sigma}$ of $X/\langle \tau\rangle$ such that $\overline{\sigma}(B)=B$. Hence, if $\tau$ were in the center of $\text{Aut}(X)$ and there were no non-trivial automorphism of $X/\langle \tau\rangle$ leaving $B$ invariant, it would follow that $\text{Aut}(X)=\langle \tau\rangle\simeq \mathbb{Z}_2$. This idea is simple and familiar to experts on the topic but there are several technical details that have to be worked out. Furthermore, even if some of these technical details are somehow expected,
as far as the author knows they have not been written down elsewhere and they are interesting on their own.

The paper is structured as follows. Section \ref{Horikawa surfaces.} contains the results about Horikawa surfaces that  will be needed throughout the note. 
In Section \ref{SectionSimpleCyclic} it is explained how to construct simple cyclic covers and to obtain information about them. An elementary criterion for a surface that can be realized as a degree $n$ simple cyclic cover to have group of automorphisms isomorphic to $\mathbb{Z}_n$ is included.
% (see Theorem \ref{AutomorphismsSimpleCyclicCovers}).
 Section \ref{AutoSection} contains a well known upper semicontinuity theorem for families of stable curves. A proof of this theorem is included for lack of a reference.
 Section \ref{SectionSimplyConnected} consists of a series of results about simply connected algebraic surfaces that will be needed to prove Theorem \ref{AutomorphismsEvenHorikawa}. Section \ref{FinalSect} is devoted to prove Theorem \ref{AutomorphismsEvenHorikawa} making use of the tools developed in the previous sections.

\section{Horikawa surfaces on the line
\texorpdfstring{$K^2=2\chi-6$}{K2}.} \label{Horikawa surfaces.} 

% Horikawa \cite{Hor1} studied minimal surfaces of general type $X$ such that $K_X^2=2\chi(\mathcal{O}_X)-6$.
% The following theorems are some of the results proved in \cite{Hor1}.
This section collects the results about Horikawa surfaces that  will be needed throughout the note.

\begin{theorem}[{{\cite[Lemma 1.1]{Hor1}}}]\label{StructureEvenHorikawa}
Let $X$ be a minimal algebraic surface with $K_X^2=2\chi(\mathcal{O}_X)-6$
and $\chi(\mathcal{O}_X) \geq 4$. Then the canonical system $|K_X|$ has no base point. Moreover,
the canonical map $\varphi_{K_X}\colon X\to \mathbb{P}^{p_g(X)-1}$ is a morphism 
of degree 2 onto a surface of degree $p_g(X)-2$ in $\mathbb{P}^{p_g(X)-1}$.
\end{theorem}

\begin{theorem}[{{\cite[Theorem 3.3, Theorem 4.1 and Theorem 7.1]{Hor1}}}]
% or {\cite[Theorem 2.5]{RanaRollenske}}}]
\label{DefClass}
Let $(K^2, \chi)$ be an admissible pair such that $K^2=2\chi-6$.
If $K^2\notin 8\cdot\mathbb{Z}$ then 
$\mathfrak{M}_{K^2, \chi}$ has a unique irreducible component. 
If $K^2\in 8\cdot\mathbb{Z}$ then $\mathfrak{M}_{K^2, \chi}=\mathfrak{M}_{K^2, \chi}^I\sqcup\mathfrak{M}_{K^2, \chi}^{II}$ has two irreducible connected components.
The image of the canonical map of a surface in $\mathfrak{M}_{K^2, \chi}^I$ is $\mathbb{F}_e$
for some $e\in\{0,2,\ldots,\frac{1}{4}K^2\}$. The image of the canonical map of a 
surface in $\mathfrak{M}_{K^2, \chi}^{II}$ is $\mathbb{F}_{\frac{1}{4}K^2+2}$ if $K^2>8$ and $\mathbb{P}^2$
or a cone over a rational curve of degree $4$ in $\mathbb{P}^4$ if $K^2=8$.
% minimal algebraic surfaces $X$ such
% that $K^2_X=K^2$ and $\chi(\mathcal{O}_X)=\chi$ have one and the same deformation type.
% If $K^2\in 8\cdot\mathbb{Z}$ then minimal algebraic surfaces $X$ such 
% that $K^2_X=K^2$ and $\chi(\mathcal{O}_X)=\chi$ have two deformation classes.
% The image of the canonical map of a surface in the first class is $\mathbb{F}_e$
% for some $e\in\{0,2,\ldots,\frac{1}{4}K^2\}$. The image of the canonical map of a 
% surface in the second class is $\mathbb{F}_{\frac{1}{4}K^2+2}$ if $K^2>8$ and $\mathbb{P}^2$
% or a cone over a rational curve of degree $4$ in $\mathbb{P}^4$ if $K^2=8$.
\end{theorem}

% \begin{remark}
% Given $k\geq 1$ we will denote by $\mathfrak{M}_{8k, 4k+3}^{I}$ (resp. $\mathfrak{M}_{8k, 4k+3}^{II}$) 
% the connected component of $\mathfrak{M}_{8k, 4k+3}$ containing  the surfaces on the first
% (resp. second) deformation class.
% \end{remark}

\section{Simple cyclic abelian covers.}\label{SectionSimpleCyclic}
Let $Y$ be a smooth surface. Suppose there exist a line bundle $L$ and an effective divisor $B$ on $Y$ such that 
$B\in |nL|$.
We denote by $V(L)$ the total space of the bundle $L$,
by $\pi\colon V(L)\to Y$ the bundle projection
and by  $t\in H^0(V(L),\pi^*L)$ the tautological section.
If $s\in H^0(Y, nL)$ is a section 
vanishing exactly along $B$, then the zero divisor $X$
of the section $t^n-\pi^*s\in H^0(V(L),\pi^*(nL))$
defines a surface in $V(L)$.
Moreover, if we denote by $\mu$ a primitive $n$-root of unity, the map $\tau\colon t\mapsto \mu\cdot t$ induces a $\mathbb{Z}_n$-action on $X$ such that $f:=\pi|_{X}$ can be realized as the quotient map $X\to X/\mathbb{Z}_n\simeq Y$.
The morphism $f\colon X\to Y$ is said to be a simple cyclic cover of degree $n$ with branch locus $B$ (see \cite[Section I.17]{Barth2004}). The set $\{L,B\}$ is known as the building data of the simple cyclic cover. 

 Simple cyclic covers are a special type of abelian cover. Let $G$ be a finite abelian group. A $G$-cover of a surface $Y$ is a finite map $f\colon X\to Y$ together with a faithful action of $G$ on $X$ such that $f$ exhibits $Y$ as $X/G$. Abelian covers in general were first studied by Pardini \cite{Par1991} but simple cyclic covers were already considered by Comessatti \cite{Comessatti}. Other references where particular types of covers were studied are 
\cite{Cata1984}, \cite{Mir1985},  \cite{PerssonDC} or \cite{Tan1991}.
 
In this note we are mainly interested in $\mathbb{Z}_2$-covers. Since every $\mathbb{Z}_2$-cover is simple cyclic and some of the results needed to prove Theorem \ref{AutomorphismsEvenHorikawa} can be easily generalized to simple cyclic covers, we will also deal with this type of covers.  

\begin{remark}\label{RedRedBuildingData} 
Let $f\colon X\to Y$ be a simple cyclic cover of degree $n$ with branch locus $B$.
 Note that if the Picard group of $Y$ has no $n$-torsion then the line bundle $L$ can be deduced from the divisor $B$. 
 In this note we are only going to consider covers of simply connected surfaces, for which the Picard group has no torsion (cf. \cite[Remark 3.10]{MendesPardini2021}).
\end{remark}

\begin{remark}
Let $f\colon X\to Y$ be a simple cyclic cover of degree $n$ with branch locus $B$.
 If $Y$ is smooth, then $X$ is smooth if and only if $B$ is smooth (cf. \cite[Section I.17]{Barth2004} or \cite[Proposition 3.1]{Par1991}) and in this case we will say that
 $f\colon X\to Y$ is a smooth simple cyclic cover.
\end{remark}

\begin{proposition}
[{{\cite[Proposition 4.2]{Par1991}}}] \label{SimpleCyclicInvariants}
Let $Y$ be a smooth surface and $f\colon X\to Y$ a smooth degree $n$ simple cyclic cover with building data 
$\{L,B\}$. Then:
\begin{equation*}
\begin{split}
K_X\equiv f^*(K_Y+(n-1)L),\\
K_X^2=n(K_Y+(n-1)L)^2,\\
p_g(X)=p_g(Y)+\sum_{i=1}^{n-1}h^0(K_Y+iL),\\
\chi(\mathcal{O}_X)=n\chi(\mathcal{O}_Y)+\frac{1}{2}\sum_{i=1}^{n-1}iL(iL+K_Y).
\end{split}
\end{equation*}
\end{proposition}

\begin{remark}\label{CanonicalImageSimpleCyclicCover}
Let $Y$ be a smooth surface and let us consider a smooth simple cyclic cover $f\colon X\to Y$ of degree $n$ with building data $\{L,B\}$. Let us assume that 
    \begin{equation*}
     h^0(K_Y)=h^0(K_Y+L)=\cdots=h^0(K_Y+(n-2)L)=0.
    \end{equation*}
    Then $K_X=f^*(K_Y+(n-1)L)$ and $p_g(X)=h^0(K_Y+(n-1)L)=:N$ by Proposition \ref{SimpleCyclicInvariants}. If we denote  by $i\colon Y\dashrightarrow \mathbb{P}^{N-1}$ the (possibly rational) map defined by the complete linear system $|K_Y+(n-1)L|$, it follows 
that $i\circ f$ is the map induced by the complete linear system $|K_X|$, i.e. it is the canonical map of $X$. In particular $i(Y)$ is the canonical image of $X$.
\end{remark}

\begin{remark}\label{AutXD}
 Let $D$ be an effective divisor on $Y$. In what follows we will denote by $\text{Aut}(Y,D)$ the subset of $\text{Aut}(Y)$ consisting
 of automorphisms $h$ of $Y$ such that $h(D)=D$.
\end{remark}

The following result is clearly inspired by \cite[Lemma 5.3]{Manetti1996} and \cite{Manetti1997}. Although it is probably well known, a proof is included for lack of a reference.

\begin{theorem}\label{AutomorphismsSimpleCyclicCovers}
 Let $Y$ be a smooth surface and $f\colon X\to Y$ a degree $n\geq 2$ simple cyclic cover with building data $\{L,B\}$
such that the branch locus $B$ is a general member  of the linear system $|nL|$. Suppose that the map defined by the complete linear system $|K_Y+(n-1)L|$
 is birational and that
 \begin{equation*}
  h^0(K_Y)=h^0(K_Y+L)=\cdots=h^0(K_Y+(n-2)L)=0.
 \end{equation*}
 If we denote by $\tau$ an order $n$ automorphism of $X$ such that $f$ can be realized as the quotient of $X$ by the action of $\langle \tau\rangle\simeq \mathbb{Z}_n$, then:
 \begin{enumerate}
  \item[i)] $\langle \tau\rangle\simeq \mathbb{Z}_n$ is a normal subgroup of $\text{Aut}(X)$;
  \item[ii)] $\langle \tau\rangle\simeq \mathbb{Z}_2$ is in the center of $\text{Aut}(X)$ if $n=2$;
  \item[iii)] $\text{Aut}(X)=\langle \tau\rangle\simeq \mathbb{Z}_n$ if $\text{Aut}(Y,B)=\{1\}$.
 \end{enumerate} 
\end{theorem}

\begin{proof} 
First of all, it follows from Remark \ref{CanonicalImageSimpleCyclicCover} that the canonical map $\Phi$ of $X$ is the composition of $f$ with the map $\Psi$ induced by the complete linear system $|K_Y+(n-1)L|$. In particular, $\Phi$ has degree $\text{deg}(\Phi)=n$ because $\Psi$ is birational. 

Let us consider the Galois group $G=\left\{g\in\text{Aut} (X):\Phi\circ g=\Phi\right\}$ of $\Phi$.
On the one hand, $\langle\tau\rangle\simeq \mathbb{Z}_n$ is contained in $G$. Given that the order of $G$ is at most $\text{deg}(\Phi)=n$, we infer that $G=\langle\tau\rangle\simeq \mathbb{Z}_n$.
On the other hand, $G$ coincides with
$\left\{g\in\text{Aut} (X):g^*C=C\text{ for every }C\in|K_X|\right\}$
% $$\bigcap_{C\in|K_X|}\text{Aut}(X,C)$$
and this is clearly a normal subgroup of $\text{Aut}(X)$.
Hence, we conclude that $\langle\tau\rangle\simeq \mathbb{Z}_n$ is a normal subgroup of $\text{Aut}(X)$. 
If $n=2$, then 
$\langle\tau\rangle\simeq \mathbb{Z}_2$ is in the center of $\text{Aut}(X)$ because order $2$ normal subgroups are always central.

That being said, we deduce from the fact that $\langle\tau\rangle\simeq \mathbb{Z}_n$ is normal in 
 $\text{Aut}(X)$
 that $h$ induces an automorphism $\overline{h}$ of 
 $X/\mathbb{Z}_n\simeq Y$ for every $h\in \text{Aut}(X)$. Moreover, denoting by $R=(f^*B)_{\text{red}}$ the ramification of $f$, I claim that $h(R)=R$ and therefore $\overline{h}(B)=B$. Indeed, let $T=\text{fix}(\tau)$ be the set of points fixed by $\tau$. Then the set $\text{fix}(h\tau h^{-1})$ of points fixed by $h\tau h^{-1}$ is equal to $h(T)$. Now, since $G$ is normal, there exists an integer $k$ coprime to $n$ such that 
 $h\tau h^{-1}=\tau^k$ and therefore $h(T)=\text{fix}(h\tau h^{-1})=\text{fix}(\tau^k)=T$. We conclude that $h(R)=R$ because $R$ is the divisorial part of $T$. As a consequence, $\overline{h}(B)=B$. If besides $\text{Aut}(Y,B)=\{1\}$,
then $\overline{h}$ is the trivial automorphism of $Y\simeq X/\mathbb{Z}_n$
 and therefore $h$ belongs to $\langle\tau\rangle\simeq \mathbb{Z}_n$. Thus, $\text{Aut}(X)=\langle\tau\rangle\simeq \mathbb{Z}_n$ in this case.
\iffalse
the ramification $R=(f^*B)_{\text{red}}$ of $f$ is contained in the set of critical points of $\Phi$ and as a consequence, $h(R)=R$ for every $h\in \text{Aut}(X)$.

Let us consider the group $G=\left\{g\in\text{Aut} (X):\Phi\circ g=\Phi\right\}.$

On the one hand, $\Phi$ has degree $n$ because $\Psi$ is birational and therefore the order of $G$ is at most $n$. Furthermore, $\langle\tau\rangle\simeq \mathbb{Z}_n$ is contained in $G$, so we infer that $G=\langle\tau\rangle\simeq \mathbb{Z}_n$.

On the other hand, $G$ coincides with
$\left\{g\in\text{Aut} (X):g^*C=C\text{ for every }C\in|K_X|\right\}$
% $$\bigcap_{C\in|K_X|}\text{Aut}(X,C)$$
and this is clearly a normal subgroup of $\text{Aut}(X)$.

Hence, we conclude that $\langle\tau\rangle\simeq \mathbb{Z}_n$ is a normal subgroup of $\text{Aut}(X)$. If $n=2$, then 
$\langle\tau\rangle\simeq \mathbb{Z}_2$ is in the center of $\text{Aut}(X)$ because order $2$ normal subgroups are always central.

In addition, we deduce from the fact that $\langle\tau\rangle\simeq \mathbb{Z}_n$ is normal in 
 $\text{Aut}(X)$
 that $h$ induces an automorphism $\overline{h}$ of 
 $X/\mathbb{Z}_n\simeq Y$. Moreover, $\overline{h}(B)=B$ because $h(R)=R$. If besides $\text{Aut}(Y,B)=\{1\}$,
then $\overline{h}$ is the trivial automorphism of $Y\simeq X/\mathbb{Z}_n$
 and therefore $h$ belongs to $\langle\tau\rangle\simeq \mathbb{Z}_n$. Thus, $\text{Aut}(X)=\langle\tau\rangle\simeq \mathbb{Z}_n$ in this case.
\fi
\end{proof}

\section{Automorphisms of a family of stable curves.}\label{AutoSection}

Let $f\colon X\to B$ be a fibration, i.e. a proper and surjective morphism with connected fibers from a surface $X$ to a smooth connected curve $B$. As in \cite{FanPar} we will denote by 
$\text{Aut}_{X/B}$ the $B$-scheme of automorphisms of the fibers of $f$. In particular, the fiber of 
$\text{Aut}_{X/B}\to B$ over $b\in B$ is isomorphic to the
group of automorphisms of the fiber of $f$ over $b$.
The aim of this section is to prove the following:
\begin{theorem}
\label{SemicontResult}
 Let $\mathcal{X}\to \Delta$ be a genus $g\geq 2$ fibration 
 over the unit disc $\Delta\subset  \mathbb{C}$ such that $\mathcal{X}$ is smooth and $\mathcal{X}_t$ is a stable curve for every $t\in\Delta$.
 Then 
 $\text{Aut}_{\mathcal{X}/\Delta}\to \Delta$ is proper and 
 the map sending $t\in\Delta$ to the order of the group of automorphisms of $\mathcal{X}_t$ is an upper semicontinuous function.
\end{theorem}
This result is well known but we include a proof, that was pointed out to the author by Rita Pardini, for lack of a reference.
\begin{proof}[Proof of Theorem \ref{SemicontResult}]
 Denoting $\Delta^*=\Delta\setminus\{0\}$ and $\mathcal{X}^*=\mathcal{X}\setminus\mathcal{X}_0$, the properness of $\text{Aut}_{\mathcal{X}/\Delta}\to \Delta$ will follow if we show that every section of $\text{Aut}_{\mathcal{X}^*/\Delta^*}\to \Delta^*$ can be extended uniquely to a section of $\text{Aut}_{\mathcal{X}/\Delta}\to \Delta$ by the 
 valuative criterion of properness. Let $\sigma$ be a section of $\text{Aut}_{\mathcal{X}^*/\Delta^*}\to \Delta^*$ and
 denote by $\hat{\sigma}\colon\mathcal{X}\dashrightarrow \mathcal{X}$ the birational map induced by $\sigma$. Suppose that $\hat{\sigma}$ cannot be extended to a morphism and consider a minimal sequence of blow-ups $\varepsilon\colon\mathcal{X}'\to \mathcal{X}$  such that $f:=\hat{\sigma}\circ \varepsilon$ is a morphism. Write
  \begin{equation*}
   K_{\mathcal{X}'}=f^*K_{\mathcal{X}}+\sum_{i=1}^r E_i  
   \end{equation*}
where each $E_i, i\in\{1,\ldots, r\}$ is an $f$-exceptional curve. By construction the last irreducible $(-1)$-curve $\Gamma$ arising from $\varepsilon$ is not contracted by $f$ and so $\Gamma\sum_{i=1}^r E_i\geq 0$ because $\Gamma$ is not a component of $\sum_{i=1}^r E_i$. In addition, since
$\mathcal{X}_t$ is a stable curve for every $t\in\Delta$ the divisor $K_{\mathcal{X}}$ is relatively nef and
$\Gamma f^*K_{\mathcal{X}}\geq0$. We conclude that:
\begin{equation*}
 -1= \Gamma K_{\mathcal{X}'}= \Gamma f^*K_{\mathcal{X}}+\Gamma\sum_{i=1}^r E_i \geq 0,
\end{equation*}
which is a contradiction. Hence we can assume $\hat{\sigma}$ to be a morphism because it can be extended to $\mathcal{X}$.
Moreover, this extension is unique because $\mathcal{X}^*$ is an open and dense subset of $\mathcal{X}$.
Since $\text{Aut}_{\mathcal{X}/\Delta}\to \Delta$ has finite fibers because $\mathcal{X}_t$ is a stable curve for every $t\in\Delta$, there exists an integer $m>0$ such that 
$(\hat{\sigma}|_{\mathcal{X}^*})^m$ is the trivial automorphism and therefore $\hat{\sigma}^m$ is the trivial automorphism on $\mathcal{X}$. In particular, $\hat{\sigma}$ is an automorphism and $\sigma$ can be extended uniquely to a section of $\text{Aut}_{\mathcal{X}/\Delta}\to \Delta$. It follows that $\text{Aut}_{\mathcal{X}/\Delta}\to \Delta$ is proper.

Write $\text{Aut}_{\mathcal{X}/\Delta}=Y\sqcup Z$ where $Y$ is the union of the $1$-dimensional components of $\text{Aut}_{\mathcal{X}/\Delta}$ and $Z$ consists of isolated points. Then the restriction $Y\to \Delta$ of $\text{Aut}_{\mathcal{X}/\Delta}\to \Delta$ to $Y$ is flat by \cite[Proposition III.9.7]{Hartshorne1977}. Moreover, 
$Y\to\Delta$ is étale. Indeed, by
\cite[Exercise III.10.3]{Hartshorne1977} it suffices to show that $Y\to\Delta$ is unramified, but this is clear 
%it suffices to show that 
since $\text{Aut}(\mathcal{X}_t)$ is reduced (because it is a complex group scheme) and finite (because $\mathcal{X}_t$ is a stable curve) for every $t\in\Delta$.
In particular, the cardinality $\# Y_t$ of the fibers of $Y\to \Delta$ is constant for every $t\in \Delta$   (cf. \cite[Corollary to Proposition 3.13]{Fischer}). The semicontinuity of the map sending $t\in\Delta$ to the order
$|\text{Aut}(\mathcal{X}_t)|$ of the group of automorphisms of $\mathcal{X}_t$ follows taking into account that:
\begin{enumerate}
 \item[i)] $|\text{Aut}(\mathcal{X}_t)|=\# Y_t$ if $Z$ is disjoint from the fiber of $\text{Aut}_{\mathcal{X}/\Delta}\to \Delta$ over $t\in \Delta$;
 \item[ii)] $|\text{Aut}(\mathcal{X}_t)|>\# Y_t$ otherwise.
\end{enumerate}
\end{proof}

\section{Simply connected algebraic surfaces.}\label{SectionSimplyConnected}

In this section we gather some results about simply connected algebraic surfaces that will be needed to prove 
Theorem \ref{AutomorphismsEvenHorikawa}. They will come up as corollaries of the following:

\begin{theorem}\label{Lefschetz}
 Let $S$ be a smooth and simply connected algebraic surface and $\Lambda$ a very ample linear system on $S$ such that the general 
 curve in $\Lambda$ has genus $\geq 3$ and is non-hyperelliptic. 
 Then there exists a dense open subset of  $\Lambda$ consisting of curves without non-trivial automorphisms.
\end{theorem}

\begin{proof}
 By \cite[10.6.18]{RMFEM} all but at most finitely many of the curves in a general $1$-dimensional linear subspace of $\Lambda$ have no non-trivial automorphisms. Therefore a general element of $\Lambda$ has no non-trivial automorphism.
\end{proof}

As a consequence we obtain the following (see Remark \ref{AutXD}):

\begin{corollary}\label{Fe}
 Let $\mathbb{F}_e$ be the Hirzebruch surface with negative section $\Delta_0$
 of self-intersection $(-e)$ and fiber $F$. Then a general member $D$ of the linear system $|a\Delta_0+bF|$ with  
 $a>2, b>\text{max}\{ae, (a-1)e+2\}$ satisfies $\text{Aut}(\mathbb{F}_e,D)=\{1\}$.
\end{corollary}

\begin{proof}
 First of all, the linear system $|a\Delta_0+bF|$ is very ample by \cite[Corollary V.2.18]{Hartshorne1977}. On the other hand, by the 
 Adjunction Formula
 \begin{equation*}
  K_D\equiv (K_{\mathbb{F}_e}+D)|_{D}\equiv ((a-2)\Delta_0+(b-2-e)F)|_D.
 \end{equation*}
Since $(a-2)\Delta_0+(b-2-e)F$ is very ample again by \cite[Corollary V.2.18]{Hartshorne1977}, 
we have that $K_D$ is very ample and therefore $D$ is a non-hyperelliptic
 curve of genus greater than $2$. Therefore we can apply Theorem \ref{Lefschetz} 
 to the smooth and simply connected surface $\mathbb{F}_e$ and the linear system 
 $|a\Delta_0+bF|$ to conclude that $D$ has no non-trivial automorphisms.
 
  Let us consider $h\in \text{Aut}(\mathbb{F}_e, D)$. We are going to show that it is necessarily the trivial automorphism of $\mathbb{F}_e$. Firstly, $h|_D$ is an automorphism of $D$ and therefore it is trivial. Hence $D$ belongs to the fixed locus of $h$ and a general fiber $F$ of $\mathbb{F}_e$ has at least $FD=a$ points fixed by $h$. Then $F$ and $h(F)$ are irreducible curves such that  $F\cdot h(F)\geq FD=a>2$ and $0=F^2=h(F)^2$. I claim that this yields $h(F)=F$. Indeed, if $e\neq 0$ the only irreducible self-intersection $0$ curves of $\mathbb{F}_e$ are the elements of $|F|$. If $e=0$ the only irreducible self-intersection $0$ curves of $\mathbb{F}_e$ are the elements of the pencils $|F|$ and $|\Delta_0|$. In both cases the fact that $F$ contains $a>2$ fixed points implies $h(F)=F$. It follows that 
 $h|_F$ is an automorphism of $F\simeq \mathbb{P}^1$ with at least $a>2$ fixed points. We conclude that $h|_F$ is the trivial automorphism of $F$. As a consequence $h$ is trivial in an open and dense subset of $\mathbb{F}_e$ and it must 
 be the trivial automorphism.  
\end{proof}

\begin{corollary}\label{F2k+2}
Let $\mathbb{F}_{2k+2}$ be the Hirzebruch surface with negative section $\Delta_0$
 of self-intersection $-(2k+2)$ and fiber $F$.
If $k\geq 1$ then a general member $D$ of the linear system $|\mathcal{O}_{\mathbb{F}_{2k+2}}(5\Delta_0+10(k+1)F)|$  
 satisfies $\text{Aut}(\mathbb{F}_{2k+2},D)=\{1\}$. In particular $\text{Aut}(\mathbb{F}_{2k+2},\Delta_0+D)=\{1\}$
\end{corollary}
\begin{proof}
Let us denote by $|\text{Aut}(C)|$ the order of the group of automorphisms of a member $C$ of the linear system $|5\Delta_0+10(k+1)F|$.
By \cite[Corollary 4.5]{FanPar} the map $C\mapsto |\text{Aut}(C)|$
 is an upper semicontinuous function on the subset of $|5\Delta_0+10(k+1)F|$ consisting of smooth divisors.
Hence, there will be an open and dense subset of curves in $|5\Delta_0+10(k+1)F|$ without non-trivial automorphisms as long as the set of smooth curves with this property is not empty. Moreover, arguing like we did in the proof of Corollary \ref{Fe} we can show that a divisor $C\in|5\Delta_0+10(k+1)F|$ without non-trivial automorphisms satisfies $\text{Aut}(\mathbb{F}_{2k+2},C)=\{1\}$. Therefore the result will follow if we find a smooth divisor $C\in|5\Delta_0+10(k+1)F|$ without  non-trivial automorphisms.

Now, by Theorem \ref{Lefschetz} a general curve $B'\in|4\Delta_0+10(k+1)F|$ has no non-trivial automorphisms. We can assume $B'$ to be smooth and irreducible and to intersect $\Delta_0$ transversally in $2(k+1)$ different points. Then it is clear that $B=\Delta_0+B'\in |5\Delta_0+10(k+1)F|$ has trivial group of automorphisms. 
Denote by $\Delta\subset\mathbb{C}$ the unit disc.
We can consider a family $p\colon \mathcal{X}\to \Delta$ of curves of $|5\Delta_0+10(k+1)F|$ such that
 $B$ is the central fiber, the rest of fibers are smooth and
 $\mathcal{X}$ is smooth. By Theorem \ref{SemicontResult}
 the cardinality of the fiber of 
 $\text{Aut}_{\mathcal{X}/\Delta}\to \Delta$
 is an upper semicontinuous function. Hence, the general fiber of $p$ has trivial group of automorphisms because $B$ does.
 \end{proof}

\begin{corollary}\label{P2}
 A general member $D$ of the linear system $|\mathcal{O}_{\mathbb{P}^2}(d)|$ with  
 $d\geq 4$ satisfies $\text{Aut}(\mathbb{P}^2,D)=\{1\}$.
\end{corollary}

\begin{proof}
 The proof of this result is analogous to the proof of Corollary \ref{Fe}.
\end{proof}

\begin{corollary}\label{ParticularAutGroups}
 Let $(S,D)$ be one of the following pairs:
 \begin{enumerate}
  \item[i)] $S$ is the Hirzebruch surface $\mathbb{F}_e$ with $e\geq 0$ and $D$ is a general member of the linear system $|a\Delta_0+bF|$ with even $a\geq 6$ and even  $b>\text{max}\{ae$, $ (a-1)e+2, (a-2)e+4\}$.
  \item[ii)] $S$ is the Hirzebruch surface $\mathbb{F}_{2k+2}$ with $k\geq 2$ and $D$ is a general member of the linear system $|\mathcal{O}_{\mathbb{F}_{2k+2}}(6\Delta_0+10(k+1)F)|$.
  \item[iii)] $S=\mathbb{P}^2$ and $D$ is a general member of the linear system $|\mathcal{O}_{\mathbb{P}^2}(d)|$ with even $d\geq 8$.
 \end{enumerate}
Then $\text{Aut}(X)\simeq\mathbb{Z}_2$ where $X\to S$ is a $\mathbb{Z}_2$-cover of $S$ branched along $D$.
\end{corollary}

\begin{proof}
We know that $\text{Aut}(S,D)=\{1\}$ by Corollary \ref{Fe}, Corollary \ref{F2k+2} or Corollary \ref{P2}
depending on whether we are in case i), ii) or iii) respectively. The result follows now from Theorem \ref{AutomorphismsSimpleCyclicCovers}.iii) taking into account that $h^0(K_S)=0$ and $K_S+\frac{1}{2}D$ is very ample.
\end{proof}

\section{Group of automorphisms of Horikawa surfaces.}\label{FinalSect}

The aim of this section is to prove Theorem \ref{AutomorphismsEvenHorikawa}. We will need the following:

\begin{lemma}\label{AutoLemma}
 Let $(K^2,\chi)$ be an admissible pair such that $K^2=2\chi-6$. Suppose there exists a smooth surface $X\in \mathfrak{M}_{K^2, \chi}$ whose group of automorphisms is isomorphic to $\mathbb{Z}_2$ and denote by $\mathfrak{M}$ the irreducible component  of $\mathfrak{M}_{K^2, \chi}$ containing $X$. Then the group of automorphisms of a general surface in  $\mathfrak{M}$ is isomorphic to $\mathbb{Z}_2$.
\end{lemma}

\begin{proof}
Let us denote by $|\text{Aut}(Y)|$ the order of the group of automorphisms of a surface $Y\in \mathfrak{M}_{K^2, \chi}$.
  By \cite[Corollary 4.5]{FanPar} the map $Y\mapsto |\text{Aut}(Y)|$
 is an upper semicontinuous function on the subset of $\mathfrak{M}$ consisting of smooth surfaces. The result follows taking into account that the group of automorphisms of every Horikawa surface has a subgroup isomorphic to $\mathbb{Z}_2$ (see Section \ref{Intro}).
\end{proof}

We are ready to prove Theorem \ref{AutomorphismsEvenHorikawa}.

\begin{proof}[Proof of Theorem \ref{AutomorphismsEvenHorikawa}]
By Lemma \ref{AutoLemma} finding a smooth surface with group of automorphisms isomorphic to $\mathbb{Z}_2$ in each irreducible component of 
 $\mathfrak{M}_{K^2, \chi}$ suffices to prove Theorem \ref{AutomorphismsEvenHorikawa}. The rest of the proof is devoted to construct such surfaces distinguishing the case $\chi$ even from the case $\chi$ odd. Notice that if $\chi$ is even then $K^2\notin 8\cdot \mathbb{Z}$ and $\mathfrak{M}_{K^2, \chi}$ has a unique irreducible component by Theorem \ref{DefClass}. If $\chi$ is odd it may happen that $K^2\in 8\cdot \mathbb{Z}$ and $\mathfrak{M}_{K^2, \chi}$ has two irreducible components by Theorem \ref{DefClass}.

 Let us begin by assuming that $\chi$ is odd. We denote by $\Delta_0$ and $F$ the two classes of fibers of 
$\mathbb{P}^1\times \mathbb{P}^1$ and
 we consider a $\mathbb{Z}_2$-cover $f\colon X\to \mathbb{P}^1\times \mathbb{P}^1$ whose branch locus $B$
 is a general member of the linear system $|6\Delta_0+(\chi+1)F|$. The formulas for simple cyclic covers (see Proposition \ref{SimpleCyclicInvariants}) yield $K^2_X=K^2$ and $\chi(\mathcal{O}_X)=\chi$. Moreover, $K_X$ is ample because it is the pullback via $f$ of the ample divisor $\Delta_0+\frac{\chi-3}{2}F$. Hence, $X\in\mathfrak{M}_{K^2, \chi}$ and 
 $\text{Aut}(X)=\mathbb{Z}_2$ by Corollary \ref{ParticularAutGroups}.
 Furthermore, if 
 $K^2=8k$ for some integer $k\geq 1$ the surface $X$ belongs to  $\mathfrak{M}^I_{8k, 4k+3}$ (see Theorem \ref{DefClass})
 because its canonical image is $\mathbb{P}^1\times \mathbb{P}^1$. Indeed, since $\Delta_0+\frac{\chi-3}{2}F$ is not only ample but very ample and $h^0(K_{\mathbb{P}^1\times \mathbb{P}^1})=0$ the canonical image of $X$ is $\mathbb{P}^1\times \mathbb{P}^1$ by Remark \ref{CanonicalImageSimpleCyclicCover}. 
 
 In order to construct a surface with group of automorphisms isomorphic to $\mathbb{Z}_2$ in  $\mathfrak{M}^{II}_{8k, 4k+3}$
 when $k\geq 2$ it suffices to consider a
 $\mathbb{Z}_2$-cover $f\colon X\to \mathbb{F}_{2k+2}$ of the Hirzebruch surface $\mathbb{F}_{2k+2}$
 with negative section $\Delta_0$ of self-intersection $-(2k+2)$ and fiber $F$.
 The branch locus $B$
 consists of $\Delta_0$ plus a general element of the linear system $|5\Delta_0+10(k+1)F|$.
 The formulas for simple cyclic covers (see Proposition \ref{SimpleCyclicInvariants}) yield $K^2_X=8k$ and $\chi(\mathcal{O}_X)=4k+3$. Moreover, $K_X$ is ample because it is the pullback via $f$ of the ample divisor $\Delta_0+(3k+1)F$. Hence, $X\in\mathfrak{M}_{8k, 4k+3}$ and 
 $\text{Aut}(X)=\mathbb{Z}_2$ by Corollary \ref{ParticularAutGroups}.
 The surface $X$ belongs to $\mathfrak{M}^{II}_{8k, 4k+3}$ (see Theorem \ref{DefClass})
 because its canonical image is 
 $\mathbb{F}_{2k+2}$. Indeed, since $\Delta_0+(3k+1)F$ is not only ample but very ample and $h^0(K_{\mathbb{F}_{2k+2}})=0$ the canonical image of $X$ is $\mathbb{F}_{2k+2}$ by Remark \ref{CanonicalImageSimpleCyclicCover}. 
 
  If $k=1$ in the previous example, the canonical divisor of $X$ is not ample. Therefore we need to construct another surface for this case.
 Let us consider a $\mathbb{Z}_2$-cover $f\colon X\to \mathbb{P}^2$ whose branch locus $B$
 is a general member of the linear system $|\mathcal{O}_{\mathbb{P}^2}(10)|$. 
 The formulas for simple cyclic covers (see Proposition \ref{SimpleCyclicInvariants}) yield $K^2_X=8$ and $\chi(\mathcal{O}_X)=7$. Moreover, $K_X$ is ample because it is the pullback via $f$ of the ample line bundle $\mathcal{O}_{\mathbb{P}^2}(2)$. Hence, $X\in\mathfrak{M}_{8, 7}$ and 
 $\text{Aut}(X)=\mathbb{Z}_2$ by Corollary \ref{ParticularAutGroups}.
 The surface $X$ belongs to $\mathfrak{M}^{II}_{8, 7}$ (see Theorem \ref{DefClass})
 because its canonical image is 
 $\mathbb{P}^2$. Indeed, since $\mathcal{O}_{\mathbb{P}^2}(2)$ is not only ample but very ample and $h^0(K_{\mathbb{P}^2})=0$ the canonical image of $X$ is $\mathbb{P}^2$ by Remark \ref{CanonicalImageSimpleCyclicCover}.

 Let us assume now that $\chi>4$ is even. We consider a $\mathbb{Z}_2$-cover $f\colon X\to \mathbb{F}_1$
 of the Hirzebruch surface $\mathbb{F}_{1}$
 with negative section $\Delta_0$ of self-intersection $(-1)$ and fiber $F$. The branch locus $B$
 is a general element of the linear system $|6\Delta_0+(\chi+4)F|$.
 The formulas for simple cyclic covers (see Proposition \ref{SimpleCyclicInvariants}) yield $K^2_X=K^2$ and $\chi(\mathcal{O}_X)=\chi$. Moreover, $K_X$ is ample because it is the pullback via $f$ of the ample divisor $\Delta_0+\frac{\chi-2}{2}F$. Hence, $X\in\mathfrak{M}_{K^2, \chi}$ and 
 $\text{Aut}(X)=\mathbb{Z}_2$ by Corollary \ref{ParticularAutGroups}.

 If $\chi=4$ in the previous example, the canonical divisor of $X$ is not ample. Therefore we need to construct another surface for this case.
 Let us consider a $\mathbb{Z}_2$-cover $f\colon X\to \mathbb{P}^2$ whose branch locus $B$
 is a general element of the linear system $|\mathcal{O}_{\mathbb{P}^2}(8)|$. 
  The formulas for simple cyclic covers (see Proposition \ref{SimpleCyclicInvariants}) yield $K^2_X=2$ and $\chi(\mathcal{O}_X)=4$. Moreover, $K_X$ is ample because it is the pullback via $f$ of the ample line bundle $\mathcal{O}_{\mathbb{P}^2}(1)$. Hence, $X\in\mathfrak{M}_{2, 4}$ and 
 $\text{Aut}(X)=\mathbb{Z}_2$ by Corollary \ref{ParticularAutGroups}.
\end{proof}

\begin{remark}
 There are many examples of
 Horikawa surfaces whose group of automorphisms is not isomorphic to $\mathbb{Z}_2$. Indeed, let $(K^2,\chi)$ be an admissible pair such that $K^2=2\chi-6$. Then:
 \begin{enumerate}
  \item[-] Every irreducible component of $\mathfrak{M}_{K^2, \chi}$ contains surfaces whose group of automorphisms has a subgroup isomorphic to $\mathbb{Z}_2^2$ by \cite[Theorem 1]{Lorenzo2021}.
  \item[-] Every irreducible component of $\mathfrak{M}_{K^2, \chi}$ contains surfaces whose group of automorphisms has a subgroup isomorphic to $\mathbb{Z}_3$ by \cite[Theorem 1.1]{Lorenzo2021Z3}.
 \end{enumerate}
\end{remark}

\noindent \begin{acknowledgements}
 The author is deeply indebted to Margarida Mendes Lopes for all her help. The author also thanks Rita Pardini for pointing out some mistakes and showing him how to prove the upper semicontinuity result for families of stable curves. Finally, the author would like to express
his gratitude to the anonymous reviewer for her/his thorough reading of the paper and suggestions.
\end{acknowledgements}

\bibliographystyle{plain}      
\bibliography{GAHS}
%\vspace{5mm}

\newpage

\noindent Vicente Lorenzo \footnote{The author is a research support technician at Universidad Carlos III de Madrid. The author did the main part of this work as a Ph.D student of the Department of Mathematics and
Center for Mathematical Analysis, Geometry and Dynamical Systems of Instituto Superior T\'{e}cnico,
Universidade de Lisboa and was supported by Fundac\~{a}o para a Ci\^{e}ncia e a Tecnologia (FCT), Portugal through
the program Lisbon Mathematics PhD (LisMath), scholarship  FCT - PD/BD/128421/2017 and
projects UID/MAT/04459/2019 and UIDB/04459/2020.}\\
Telematic Engineering Department\\
Universidad Carlos III de Madrid\\
Avenida de la Universidad 30\\
Leganés (Madrid), Spain\\
\textit{E-mail address: }{vlorenzogarcia@gmail.com}\\
https://orcid.org/0000-0003-2077-6095

\end{document}